\documentclass[reqno]{amsart}
\usepackage{amsmath, amssymb, amsthm, epsfig}
\usepackage{hyperref, latexsym}
\usepackage{url}
\usepackage[mathscr]{euscript}

\usepackage{color}
\usepackage{fullpage} 
\usepackage{setspace}

\def\today{\ifcase\month\or
  January\or February\or March\or April\or May\or June\or=
  July\or August\or September\or October\or November\or December\fi
  \space\number\day, \number\year}

 \newtheorem{theorem}{Theorem}
  
 \newtheorem{lemma}[theorem]{Lemma}
 
 \newtheorem{corollary}[theorem]{Corollary}
 \theoremstyle{definition}

 \theoremstyle{remark}

 \newcommand{\R}{\mathbb{R}}
  
 \newcommand{\N}{\mathbb{N}}
 
 \newcommand{\Z}{\mathbb{Z}}

 \newcommand{\dy}{\text{\rm d}y}

\newcommand{\wt}{\widetilde}

\newcommand{\var}{{\rm Var\,}}

\begin{document}

\title[Sharp inequalities for discrete maximal operators]{Sharp inequalities for the variation \\ of the discrete maximal function}
\author[Carneiro and Madrid]{%Emanuel Carneiro and 
Jos\'e Madrid
}
\date{\today}

\address{IMPA - Instituto de Matem\'{a}tica Pura e Aplicada, Estrada Dona Castorina 110, Rio de Janeiro - RJ, Brazil, 22460-320.}
%\email{carneiro@impa.br}
\email{josermp@impa.br}

\allowdisplaybreaks
\numberwithin{equation}{section}

\maketitle

\begin{abstract}In this paper %we study the regularity properties of fractional maximal operators acting on $BV$-functions and $W^{1,1}$-functions. 
we establish new optimal bounds for the derivative of some discrete maximal functions, both in the centered and uncentered versions. In particular, we solve a question originally posed by Bober, Carneiro, Hughes and Pierce in \cite{BCHP}.
\end{abstract}

\section{Introduction}
\subsection{Background} Let $M$ denote the centered Hardy-Littlewood maximal operator on $\R^d$, i.e. for $f \in L^1_{loc}(\R^d)$,
\begin{equation}\label{Intro_max}
Mf(x) = \sup_{r >0} \frac{1}{m(B_r(x))} \int_{B_r(x)} |f (y)|\,\dy\,,
\end{equation}
where $B_r(x)$ is the ball centered at $x$ with radius $r$ and $m(B_r(x))$ is its $d$-dimensional Lebesgue measure. One of the classical results in harmonic analysis states that $M:L^p(\R^d) \to L^p(\R^d)$ is a bounded operator for $1<p \leq \infty$. For $p=1$ we have $M: L^1(\R^d) \to L^{1,\infty}(\R^d)$ bounded. In 1997,  Kinunnen \cite{Ki} showed that $M: W^{1,p}(\R^d) \to W^{1,p}(\R^d)$ is bounded for $1 < p \leq \infty$, and that was the starting point on the study of the regularity of maximal operators acting on Sobolev functions. This result was later extended to multilinear, local and fractional contexts in \cite{CM, KL, KiSa}. Due to the lack of reflexivity of $L^1$, results for $p=1$ are subtler. For instance, in \cite[Question 1]{HO}, Haj\l asz and Onninen asked whether the operator $f \mapsto |\nabla Mf|$ is bounded from $W^{1,1}(\R^d)$ to $L^1(\R^d)$. Progress on this question (and its variant for BV-functions) has been restricted to dimension $d=1$.

\smallskip

Let $\wt{M}$ denote the uncentered maximal operator (defined similarly as in \eqref{Intro_max}, with the supremum taken over all balls containing the point $x$ in its closure). Refining the work of Tanaka \cite{Ta}, Aldaz and P\'{e}rez L\'{a}zaro \cite{AP} showed that if $f$ is of bounded variation then $\wt{M}f$ is absolutely continuous and
\begin{equation}\label{Intro_AP}
\var \wt{M}f \leq \var f,
\end{equation}
where $\var f$ denotes the total variation of $f$. Observe that inequality \eqref{Intro_AP} is sharp. More recently, Kurka \cite{Ku} considered the centered maximal operator in dimension $d=1$ and proved that 
\begin{equation}\label{Intro_Ku}
\var Mf \leq 240,004\, \var f.
\end{equation}
It is currently unknown if one can bring down the value of such constant to $C=1$ in the centered case. Other interesting works related to this theory are \cite{ACP, CFS, CS, HM, Lu1, St}.

\subsection{Discrete setting} In this paper we consider issues of similar flavor, now in the discrete setting. Let us start with some definitions. 

\smallskip

We denote  a vector $\vec{n} \in \Z^d$ by $\vec{n} = (n_1, n_2, \ldots, n_d)$. For a function $f:\mathbb{Z}^{d}\rightarrow \mathbb{R}$ %(or, in general, for a vector-valued function $f:\mathbb{Z}^{d}\rightarrow \mathbb{R}^m$) 
we define its $\ell^{p}$-norm as usual:
\begin{equation}\label{Intro_l_p_norm}
\|f\|_{\ell^{p}{( \Z^{d})}}= \left(\sum_{\vec n\in \Z^{d}} {|f(\vec n)|^{p}}\right)^{1/p},
\end{equation}
if $1\leq p<\infty$, and
\begin{equation*}
\|f\|_{\ell^{\infty}{(\Z^{d})}}= \sup_{\vec n\in\Z^{d} }{|f(\vec n)|}.
\end{equation*}
We define its total variation $\var f$ by
$$
\var f= \sum_{i=1}^d \sum_{\vec n \in \Z^d} \big| f(\vec n+\vec e_{i})-f(\vec n)\big|,
$$
where $\vec e_{i}=(0,0,\ldots,1,\ldots,0)$ is the canonical $i-$th base vector. Also, we say that a function $f:\Z^{d}\to\R$ is a {\it delta function} if
there exist $\vec p\in\Z^{d}$ and $k\in\R$, such that
$$
f(\vec p)=k\ \ \ \text{and}\ \ \ f(\vec n)=0 \ \ \forall\ \vec n \in\Z^{d}\setminus\{\vec p\}.
$$

\subsubsection{A sharp inequality in dimension one} For $f:\Z\to \R$ we define its centered Hardy-Littlewood maximal function $Mf :\Z \to \R^+$ as 
\begin{equation*}
Mf(n) = \sup_{r  \in \Z^+} \frac{1}{(2r+1)} \sum_{k=-r}^r |f(n+k)|,
\end{equation*}
while the uncentered maximal function $\wt{M}f :\Z \to \R^+$ is given by
\begin{equation*}
\wt{M}f(n) = \sup_{r,s  \in \Z^+} \frac{1}{(r +s +1)} \sum_{k=-r}^s |f(n+k)|.
\end{equation*}
In \cite{BCHP}, Bober, Carneiro, Hughes and Pierce proved the following inequalities
\begin{equation}\label{obj 0}
\var \wt Mf \leq  \var f \leq 2\|f\|_{\ell^{1}(\Z)}
\end{equation}
and
\begin{equation}\label{obj 1}
\var Mf\leq \left(2+\frac{146}{315}\right) \|f\|_{\ell^{1}(\Z)}. 
\end{equation}
The leftmost inequality in \eqref{obj 0} is the discrete analogue of \eqref{Intro_AP}. The rightmost inequality in \eqref{obj 0} is simply the triangle inequality. Both inequalities in \eqref{obj 0} are in fact sharp (e.g. equality is attained if $f$ is a delta function). On the other hand, inequality \eqref{obj 1} is not optimal, and it was asked in \cite{BCHP} whether the sharp constant for \eqref{obj 1} is in fact $C=2$. Our first result answers this question affirmatively, also characterizing the extremal functions.

\begin{theorem}\label{lim d=1 C=2}
Let $f:\Z\to\R$ be a function in $\ell^{1}(\Z).$ Then
\begin{equation}\label{main theo cent d=1}
\var Mf\leq 2\,\|f\|_{\ell^{1}(\Z)},
\end{equation}
and the constant $C=2$ is the best possible. Moreover, the equality is attained if and only if $f$ is a delta function.
\end{theorem}  

\noindent {\sc{Remark:}} In \cite{Te}, Temur proved the analogue of \eqref{Intro_Ku} in the discrete setting, i.e. 
\begin{equation}\label{Te_Var}
\var Mf \leq  C \ \var f 
\end{equation}
with constant $C=(72000)2^{12}+4$. This inequality is qualitatively stronger that \eqref{main theo cent d=1} (in fact, $\var f$ should be seen as the natural object to be on the right-hand side), but it does not imply \eqref{main theo cent d=1}. By triangle inequality, inequality \eqref{main theo cent d=1} suggests that it may be possible to prove \eqref{Te_Var} with constant $C=1$, but this is currently an open problem. 

\subsubsection{Sharp inequalities in higher dimensions} We now aim to extend Theorem \ref{lim d=1 C=2} to higher dimensions. In order to do so, we first recall the notion of maximal operators associated to regular convex sets as considered in \cite{CH}.

\smallskip

Let $\Omega\subset \mathbb R^{d}$ be a bounded open convex set with Lipschitz boundary, such that $\vec 0\in$ int$(\Omega)$ and that $\pm \vec e_{i} \in \overline\Omega$ for $1 \leq i \leq d$. %(by renormalizing $\Omega$ if necessary)\footnote{This renormalization is merely aesthetic. For general $\Omega$, one has to dilate the constants appearing in our results accordingly.}. 
For $r>0$ we write 
\begin{equation*}
\overline\Omega_{r}(\vec{x}_{0}) =\big\{\vec{x} \in\mathbb R^{d}; \, r^{-1}(\vec{x}-\vec{x}_{0})\in \overline{\Omega}\big\},
\end{equation*}
and for $r=0$ we consider
\begin{equation*}
\overline\Omega_{0}(\vec{x}_{0}) =\{\vec{x}_{0}\}.
\end{equation*}
Whenever $\vec{x}_{0}=\vec 0$ we shall write $\overline\Omega_{r}=\overline\Omega_{r}\big(\vec{{0}}\big) $ for simplicity. This object plays the role of the ``ball of center $\vec x_{0}$ and radius $r$" in our maximal operators below. For instance, to work with regular $\ell^{p}-$balls, one should consider $\Omega=\Omega_{\ell^p}=\{\vec x\in\R^{d}; \|\vec x\|_{p}<1 \}$, where $\|\vec x\|_{p}=(|x_{1}|^{p}+|x_{2}|^{p}+\ldots+|x_{d}|^{p})^{\frac{1}{p}} $ for $\vec x=(x_{1},x_{2},\ldots,x_{d})\in \R^{d}$ and $1 \leq p < \infty$, and $\|\vec x\|_{\infty}=\max\{|x_1|, |x_2|, \ldots, |x_d|\}$.

\smallskip

Given  $f:\Z^d \to \R$, we denote by $A_{r}f(\vec n)$ the average of $|f|$ over the $\Omega$-ball of center $\vec n$ and radius $r$, i.e.
\begin{equation*}
A_{r}f(\vec n) = \frac{1}{N(r)}\,\sum_{\vec m\in \overline\Omega_{r} }{|f(\vec n+\vec m)}|,
\end{equation*}
where $N(\vec x,r)$ is the number of the lattice points in the set $\overline\Omega_{r}(\vec x)$\ (and $N(r):=N(\vec 0,r)$). We denote by $M_{\Omega}$ the discrete centered  maximal operator associated to $\Omega$,
\begin{equation}\label{Intro_disc_Omega_cent}
M_{\Omega}f(\vec n)=\sup_{r\geq 0}A_{r}f(\vec n)=\sup_{r\geq 0} \frac{1}{N(r)}\,\sum_{\vec m\in \overline\Omega_{r} }{|f(\vec n+\vec m)}|,
\end{equation}
and we denote by $\widetilde{M}_{\Omega}$ its uncentered version
\begin{equation}\label{Intro_disc_Omega_uncent}
\widetilde{M}_{\Omega}f(\vec n)=\sup_{\overline\Omega_{r}(\vec x_{0}) \owns \vec n}\,A_{r}f(\vec x_{0}) =\sup_{\overline\Omega_{r}(\vec x_{0}) \owns \vec n}\, \frac{1}{N(\vec x_{0},r)}\,{\sum_{\vec m\in \overline\Omega_{r}(\vec x_{0}) }{|f( \vec m)}|}.
\end{equation}
It should be understood throughout the rest of the paper that we always consider $\Omega$-balls with at least one lattice point. These convex $\Omega$--balls have roughly the same behavior as the regular Euclidean balls from the geometric and arithmetic points of view, in the sense that for large radii, the number of lattice points inside the $\Omega$-ball is roughly equal to the volume of the $\Omega$-ball (see \cite[Chapter VI \S 2, Theorem 2]{Lang}).

\smallskip

Given $1\leq p< \infty$ and $f\in \ell_{loc}^{1}(\Z^{d})$, we denote by $M_{p}$ the discrete centered maximal operator associated to $\Omega_{\ell^p}$, 
\begin{equation*}
M_{p}f(\vec n)=M_{\Omega_{\ell^p}}f(\vec n),%=\sup_{r\geq 0}A_{r}|f|(\vec n)=\sup_{r\geq 0}\frac{1}{|Q(r)|}\sum_{\vec k\in Q(r)}|f(\vec n+\vec k)|,
\end{equation*}
%where $Q(r)=[-r,r]^{d}$, and $|Q(r)|=(2\lfloor r\rfloor+1)^{d}$. 
and for $p=\infty$, we denote
$$
Mf(\vec n)=M_{\Omega_{\ell^\infty}}f(\vec n).
$$
Analogously, we denote by $\wt M_{p}f$ and $\wt Mf$ the uncentered versions of the discrete maximal operators associated to $\Omega_{\ell^p}$, for $1 \leq p \leq \infty$. Note that in dimension $d=1$ we have $M_{p}=M \ \text{and}\ \wt M_{p}=\wt M$  for all $1\leq p\leq \infty$.

\smallskip

In \cite{CH}, Carneiro and Hughes showed that, for any regular set $\Omega$ as above and $f:\Z^{d}\to\R$, there exist constants $C(\Omega,d)$ and $\wt C(\Omega,d)$ such that

\begin{equation}\label{obj 2}
\var M_{\Omega}f\leq C(\Omega,d)\|f\|_{\ell^{1}(\Z^{d})} 
\end{equation}
and
\begin{equation}\label{obj 3}
\var \wt M_{\Omega}f\leq \wt C(\Omega,d)\|f\|_{\ell^{1}(\Z^{d})}.
\end{equation}
Inequalities \eqref{obj 2} and \eqref{obj 3} were extended to a fractional setting in \cite[Theorem 3]{CMa}. Here we extend Theorem \ref{lim d=1 C=2} to higher dimensions in two distinct ways. We find the sharp form of \eqref{obj 2}, when $d\geq 1$ and $\Omega=\Omega_{\ell^1}$ (i.e. rombus), and the sharp form of \eqref{obj 3}, when $d\geq 1$ and $\Omega=\Omega_{\ell^\infty}$ (i.e. regular cubes). As we shall see below, we use different techniques in the proofs of these two extensions, taking into consideration the geometry of the chosen sets $\Omega$.

\smallskip

For $d\geq1$ and $k\geq 0$ we denote $N_{1,d}(k)=\big|\overline{(\Omega_{\ell^1})_{k}}\big|=\big|\{\vec x\in\Z^{d}; \|\vec x\|_{1}\leq k\}\big|$.  Here is our next result.

\begin{theorem}\label{main theo cent}
Let $d\geq2$ and $f:\Z^{d}\to\R$ be a function in $\ell^{1}(\Z^{d})$. Then
\begin{equation}\label{eq main theo cent}
\var M_{1}f\leq 2d\left(1+\sum_{k\geq 1}\frac{(N_{1,d-1}(k)-N_{1,d-1}(k-1))}{N_{1,d}(k)}\right)\|f\|_{\ell^{1}(\Z^{d})}=:C(d)\|f\|_{\ell^{1}(\Z^{d})},%=\left(2d+4d(d-1)\left(\left(1-\frac{1}{2^{d}}\right)\zeta(d)-1\right)\right)\|f\|_{l^{1}(\Z^{d})},
\end{equation}
%where $\zeta$ is the Riemann zeta function.
and this constant $C(d)$ is the best possible. Moreover, the equality is attained if and only if $f$ is a delta function. %In particular this is true for $d=2$.
\end{theorem}
\noindent
{\sc{Remark:}} Note that $C(d)<\infty$, because there exists a constant $C$ such that
$$
N_{1,d}(k)=Ck^{d}+O(k^{d-1}),
$$
where $C=m(\Omega_{\ell^1})$ (see \cite[Chapter VI \S 2, Theorem 2]{Lang}). Then, for sufficiently large $k$ we have
$$ \frac{N_{1,d-1}(k)-N_{1,d-1}(k-1)}{N_{1,d}(k)}\sim \frac{1}{k^{2}}. $$ In particular, for $d=2$ we obtain $$C(2)=4+8\sum_{k\geq 1}\frac{1}{k^{2}+(k+1)^{2}}.$$

\smallskip

Our proof of Theorem \ref{main theo cent} is the natural extension of the proof of Theorem \ref{lim d=1 C=2} but we decided to present Theorem \ref{lim d=1 C=2} separately since it contains the essential idea with less technical details. The next result is the sharp version of \eqref{obj 3} for the discrete uncentered maximal operator with respect to cubes (i.e. $\ell^{\infty}$-balls). This proof follows a different strategy from Theorems \ref{lim d=1 C=2} and  \ref{main theo cent}.

\begin{theorem}\label{main theo noncent}
Let $d\geq 1$ and $f:\Z^{d}\to\R$ be a function in $\ell^{1}(\Z^{d}).$ Then
\begin{equation}\label{eq noncent d>1}
\var \wt Mf\leq 2d\left(1+\sum_{k\geq 1}\frac{1}{k}\left(\left(\frac{2}{k+1}+\frac{2k-1}{k}\right)^{d-1}-\left(\frac{2k-1}{k}\right)^{d-1}\right)\right)\|f\|_{\ell^{1}(\Z^{d})}=:\wt C(d)\|f\|_{\ell^{1}(\Z^{d})},
\end{equation}
and the constant $\wt C(d)$ is the best possible. Moreover, the equality is attained if and only if $f$ is a delta function.
\end{theorem}
\noindent {\sc{Remark:}} In particular $\wt C(1)= 2$ (and we recover \eqref{obj 0}) and $\wt C(2)=12$.

\smallskip

For the proofs of these three theorems we may assume throughout the rest of the paper, without loss of generality, that $f\geq0$.

%--------------------Theo d=1 c=2---------------------------
\section{Proof of Theorem \ref{lim d=1 C=2} }
Since $f\in \ell^{1}(\Z)$, we have that for all $n\in\Z$ there exists $r_{n}\in\Z$ such that $Mf(n)=A_{r_{n}}f(n)$. We define
$$
X^{-}=\{n\in\Z; Mf(n)\geq Mf(n+1)\}\ \  \text{and}\ \ X^{+}=\{n\in\Z; Mf(n+1)>Mf(n)\}.
$$
\noindent
Then we have
\begin{eqnarray}\label{suma}
\var Mf&=&\sum_{n\in\Z}|Mf(n)-Mf(n+1)|\nonumber\\
&=&\sum_{n\in X^{-}}Mf(n)-Mf(n+1)+\sum_{n\in X^{+}}Mf(n+1)-Mf(n)\nonumber\\
&\leq&\sum_{n\in X^{-}}A_{r_{n}}f(n)-A_{r_{n}+1}f(n+1)+\sum_{n\in X^{+}}A_{r_{n+1}}f(n+1)-A_{r_{n+1}+1}f(n). 
\end{eqnarray}
Given $p\in\Z$ fixed,  we want to evaluate the
maximal contribution of $f(p)$ to the right-hand side of \eqref{suma}.

\smallskip 

\noindent{\it Case 1:} If $n\in X^{-}$ and $n\geq p$. In this situation we have that the contribution of $f(p)$ to $A_{r_{n}}f(n)-A_{r_{n}+1}f(n+1)$ is $0$ (if $p<n-r_{n}$) or $\frac{1}{2r_{n}+1}-\frac{1}{2r_{n}+3}$ (if $n-r_{n}\leq p$). In the second case we have 
\begin{eqnarray*}
\frac{1}{2r_{n}+1}-\frac{1}{2r_{n}+3}&=&\frac{2}{(2r_{n}+1)(2r_{n}+3)}\\
&\leq& \frac{2}{(2(n-p)+1)(2(n-p)+3)}\\
=&=&\frac{1}{2(n-p)+1}-\frac{1}{2(n-p)+3}.
\end{eqnarray*}
The equality is attained if and only if $r_{n}=n-p$.

\smallskip

\noindent {\it Case 2:} If $n\in X^{+}$ and $n\geq p$. Now we have that the contribution of $f(p)$ to $A_{r_{n+1}}f(n+1)-A_{r_{n+1}+1}f(n)$ is non-positive (if $p<n+1-r_{n+1}$) or $\frac{1}{2r_{n+1}+1}-\frac{1}{2r_{n+1}+3}$ (if $n+1-r_{n+1}\leq p$). In the second case we have 
\begin{eqnarray*}
\frac{1}{2r_{n+1}+1}-\frac{1}{2r_{n+1}+3}&=&\frac{2}{(2r_{n+1}+1)(2r_{n+1}+3)}\\
&\leq& \frac{2}{(2(n+1-p)+1)(2(n+1-p)+3)}\\
&=&\frac{1}{2(n+1-p)+1}-\frac{1}{2(n+1-p)+3}\\
&<&\frac{1}{2(n-p)+1}-\frac{1}{2(n-p)+3}.
\end{eqnarray*}

\smallskip

\noindent {\it Case 3:} If $n\in X^{-}$ and $n<p$. In this situation we have that the contribution of $f(p)$ to $A_{r_{n}}f(n)-A_{r_{n}+1}f(n+1)$ is non-positive (if $p>n+r_{n}$) or $\frac{1}{2r_{n}+1}-\frac{1}{2r_{n}+3}$ (if $n+r_{n}\geq p$). In the second case we have
\begin{eqnarray*}
\frac{1}{2r_{n}+1}-\frac{1}{2r_{n}+3}&=&\frac{2}{(2r_{n}+1)(2r_{n}+3)}\\
&\leq& \frac{2}{(2(p-n)+1)(2(p-n)+3)}\\
&=&\frac{1}{2(p-n)+1}-\frac{1}{2(p-n)+3}\\
&<&\frac{1}{2(p-n-1)+1}-\frac{1}{2(p-n-1)+3}.
\end{eqnarray*}

\smallskip

\noindent {\it Case 4:} If $n\in X^{+}$ and $n<p$. Now we have that the contribution of $f(p)$ to $A_{r_{n+1}}f(n+1)-A_{r_{n+1}+1}f(n)$ is either $0$ (if $p>n+1+r_{n+1}$) or $\frac{1}{2r_{n+1}+1}-\frac{1}{2r_{n+1}+3}$ (if $n+1+r_{n+1}\geq p$). In the second case we have 
\begin{eqnarray*}
\frac{1}{2r_{n+1}+1}-\frac{1}{2r_{n+1}+3}&=&\frac{2}{(2r_{n+1}+1)(2r_{n+1}+3)}\\
&\leq& \frac{2}{(2(p-n-1)+1)(2(p-n-1)+3)}\\
&=&\frac{1}{2(p-n-1)+1}-\frac{1}{2(p-n-1)+3}.
\end{eqnarray*} 
The equality is achieved if and only if $r_{n+1}=p-n-1$.

\smallskip

\noindent {\it Conclusion:} Therefore the contribution of $f(p)$ to the right-hand side of \eqref{suma} is bounded by
$$
\sum_{n\geq p}\frac{1}{2(n-p)+1}-\frac{1}{2(n-p)+3}+\sum_{n<p}\frac{1}{2(p-n-1)+1}-\frac{1}{2(p-n-1)+3}=2.
$$
As $p$ is an arbitrary point in $\Z$, this establishes \eqref{main theo cent d=1}. If $f$ is a delta function we can easily see that
$$
\var Mf =2\|f\|_{\ell^{1}(\Z)}.
$$
On the other hand, given a function $f:\Z\to\R$ such that $\var Mf=2\|f\|_{\ell^{1}(\Z)}$ and $f\geq0$, let us define $P=\{t\in\Z; f(t)\neq 0\}$. Then 
$$
\var Mf=2\sum_{t\in P}f(t),
$$
and, given $t_{1}\in P$, the contribution of $f(t_{1})$ to \eqref{suma} is 2. Therefore, by the previous analysis we note that for all $n\geq t_{1}$ we must have that $n\in X^{-}$ and $r_{n}=n-t_{1}$. If we take $t_{2}\in P$ the same should happen, which implies that $t_{1}=t_{2}$ and therefore $P=\{t_{1}\}$. This proves that $f$ is a delta function and the proof is concluded.

%-------------Main theo cent d>1--------------------------

\section{Proof of Theorem \ref{main theo cent}}

\subsection{Preliminaries} Since $f\in \ell^{1}(\Z^{d})$, we have that there exists $r_{\vec n}\in\Z$ such that $M_{1}f(\vec n)=A_{r_{\vec n}}f(\vec n)$. For all $\vec m=(m_{1},m_{2},\dots,m_{d})\in\Z^{d}$ we define
$$
|\vec m|_{1}=\sum_{i=1}^{d}|m_{i}|,
$$
and for $1\leq j\leq d$, we define
%$$
%Y_{i}=\{\vec n\in\Z^{d}; |\vec n|_{\infty}=|n_{i}|\},
%$$
$$
I_{j}=\{l\subset\Z^{d};\ l\  \text{is a line parallel to the vector}\   \vec e_{j} \},
$$
$$
X_{j}^{-}=\{\vec n\in\Z^{d}; M_{1}f(\vec n)\geq M_{1}f(\vec n+\vec e_{j})\}\ \  \text{and}\ \ X_{j}^{+}=\{\vec n\in\Z^{d}; M_{1}f(\vec n+\vec e_{j})>M_{1}f(\vec n)\}.
$$
We then have
\begin{eqnarray}\label{suma d>1}
\var M_{1}f&=&\sum_{\vec n\in\Z^{d}}\sum_{j=1}^{d}|M_{1}f(\vec n)-M_{1}f(\vec n+\vec e_{j})|\nonumber\\
&=&\sum_{j=1}^{d}\sum_{l\in I_{j}}\sum_{\vec n\in l\cap X_{j}^{-}}M_{1}f(\vec n)-M_{1}f(\vec n+\vec e_{j})+\sum_{j=1}^{d}\sum_{l\in I_{j}}\sum_{\vec n\in l\cap X_{j}^{+}}M_{1}f(\vec n+\vec e_{j})-M_{1}f(\vec n)\nonumber\\
&\leq&\sum_{j=1}^{d}\sum_{l\in I_{j}}\sum_{\vec n\in l\cap X_{j}^{-}}A_{r_{\vec n}}f(\vec n)-A_{r_{\vec n}+1}f(\vec n+\vec e_{j})\\
&&\ \ \  \ +\sum_{j=1}^{d}\sum_{l\in I_{j}}\sum_{\vec n\in l\cap X_{j}^{+}}A_{r_{\vec n+\vec e_{j}}}f(\vec n+\vec e_{j})-A_{r_{\vec n+\vec e_{j}}+1}f(\vec n)\nonumber.  
\end{eqnarray}

Fixed a point $\vec p=(p_{1},p_{2},\dots,p_{d})\in\Z^{d}$,  we want to evaluate the maximal contribution of $f(\vec p)$ to the right-hand side of \eqref{suma d>1}.

\subsection{Auxiliary results} We now prove the following lemma of arithmetic character, which will be particularly useful in the rest of the proof.

%------------------Fundamental-Lemma---------------------------- 
\begin{lemma}\label{fund lemma} If $d\geq1$, then

\begin{equation}\label{Ineq_Fund_Lem}
N_{1,d}(k)^{2}> {N_{1,d}(k+1)}N_{1,d}(k-1) \ \ \forall \ k\geq1.
\end{equation}
\end{lemma}
\begin{proof}
We prove this via induction. For $d=1$ we have that $N_{1,1}(k)=2k+1$, therefore
$$
N_{1,1}(k)^{2}=4k^{2}+4k+1>(2k+3)(2k-1)=N_{1,1}(k+1)N_{1,1}(k-1).
$$ 
Since $N_{1,d}(k)=\big|\{(x_{1},\dots,x_{d})\in\Z^{d};|x_{1}|+
 \dots +|x_{d}|\leq k\}\big|$, fixing the value of the last variable, we can verify that
\begin{equation}\label{recurrencia larga}
N_{1,d}(k)=N_{1,d-1}(k)+2\sum_{j=0}^{k-1}N_{1,d-1}(j).
\end{equation}

Now, let us assume that the result is true for $d$, i.e.
\begin{equation}\label{hip}
N_{1,d}(k)^{2}> {N_{1,d}(k+1)}N_{1,d}(k-1) \ \ \forall \ k\geq1.
\end{equation}
We want to prove that this implies that the result is also true for $d+1$. For simplicity we denote $g(k):=N_{1,d}(k)\ \ \text{and}\ \ f(k):=N_{1,d+1}(k) \ \ \text{for all}\ k\geq 0$. Thus by \eqref{hip} we have that
\begin{equation}\label{consec_hip}
\frac{g(1)}{g(0)}> \frac{g(2)}{g(1)}>\dots>\frac{g(k)}{g(k-1)}>\frac{g(k+1)}{g(k)}>\dots 
\end{equation}
and by \eqref{recurrencia larga} we have that
$$
f(k)=g(k)+2\sum_{j=0}^{k-1}g(j) \ \  \ \  \forall\ k\geq 0.
$$
The latter implies that
$$
f(k+1)-f(k)=g(k+1)+g(k)\  \ \  \  \forall\ k\geq 0. 
$$
Therefore, by \eqref{consec_hip}, we obtain that
$$
\frac{g(k+1)}{g(k)}> \frac{g(k+2)+g(k+1)}{g(k+1)+g(k)}
$$
and
$$
\frac{g(k+1)+2\sum_{j=1}^{k}g(j)}{g(k)+2\sum_{j=1}^{k}g(j-1)}>
\frac{g(k+1)}{g(k)}.
$$
Combining these inequalities we arrive at
$$
\frac{f(k+1)}{f(k)}\geq \frac{g(k+1)+2\sum_{j=1}^{k}g(j)}{g(k)+2\sum_{j=1}^{k}g(j-1)} >
\frac{g(k+1)}{g(k)}> \frac{g(k+2)+g(k+1)}{g(k+1)+g(k)}=\frac{f(k+2)-f(k+1)}{f(k+1)-f(k)},
$$
and hence
$$
\frac{f(k+1)-f(k)}{f(k)}> \frac{f(k+2)-f(k+1)}{f(k+1)}.
$$
This implies that
$$
\frac{f(k+1)}{f(k)}> \frac{f(k+2)}{f(k+1)}\ \ \forall \ k\geq 0,
$$
which establishes the desired result.
\end{proof}

\begin{corollary}\label{consec_lem}
If $d\geq1$, we have that
\begin{equation}\label{prop theo 2}
\frac{1}{N_{1,d}(k)}-\frac{1}{N_{1,d}(k+1)}> \frac{1}{N_{1,d}(k+1)}-\frac{1}{N_{1,d}(k+2)} \ \ \forall \ k\geq0.
\end{equation}
\end{corollary}
\begin{proof}
We notice that \eqref{prop theo 2} is equivalent to
$$
\frac{N_{1,d}(k+1)}{N_{1,d}(k)}+\frac{N_{1,d}(k+1)}{N_{1,d}(k+2)}>2.
$$
This follows from Lemma \ref{fund lemma} and the arithmetic mean - geometric mean inequality because
$$
\frac{N_{1,d}(k+1)}{N_{1,d}(k)}+\frac{N_{1,d}(k+1)}{N_{1,d}(k+2)}> \frac{N_{1,d}(k+2)}{N_{1,d}(k+1)}+\frac{N_{1,d}(k+1)}{N_{1,d}(k+2)}\geq 2.
$$
\end{proof}

\subsection{Proof of Theorem \ref{main theo cent}} Let us simplify notation by writing $N_{1}(k):=N_{1,d}(k)$. Given $1\leq j\leq d$, using Corollary \ref{consec_lem} we make the following observations.

\smallskip 

\noindent{\it Case 1:} If $\vec n\in X_{j}^{-}$ and $n_{j}\geq p_{j}$. In this situation we have that the contribution of $f(\vec p)$ to $A_{r_{\vec n}}f(\vec n)-A_{r_{\vec n}+1}f(\vec n+\vec e_{j})$ is non-positive (if $|\vec n-\vec p|_{1}>r_{\vec n}$) or $\frac{1}{N_{1}(r_{\vec n})}-\frac{1}{N_{1}(r_{\vec n}+1)}$ (if $|\vec n-\vec p|_{1}\leq r_{\vec n}$). In the second case we have 
\begin{eqnarray*}
\frac{1}{N_{1}(r_{\vec n})}-\frac{1}{N_{1}(r_{\vec n}+1)}
%&\leq&\frac{1}{(2(n_{i}-p_{i})+1)^{d}}-\frac{1}{(2(n_{i}-p_{i})+3)^{d}}\\
&\leq&\frac{1}{N_{1}(|\vec n-\vec p|_{1})}-\frac{1}{N_{1}(|\vec n-\vec p|_{1}+1)}\\
&=&\frac{1}{N_{1}(|\vec n-\vec p|_{1})}-\frac{1}{N(|\vec n+\vec e_{j}-\vec p|_{1})}.
\end{eqnarray*}
The equality is attained if and only if $r_{\vec n}=|\vec n-\vec p|_{1}$.

\smallskip 

\noindent{\it Case 2:} If $\vec n\in X_{j}^{+}$ and $n_{j}\geq p_{j}$. Now we have that the contribution of $f(\vec p)$ to $A_{r_{\vec n+\vec e_{j}}}f(\vec n+\vec e_{j})-A_{r_{\vec n+\vec e_{j}}+1}f(\vec n)$ is non-positive (if $|\vec n+\vec e_{j}-\vec p|_{1}>r_{\vec n+\vec e_{j}}$) or $\frac{1}{N_{1}(r_{\vec n+\vec e_{j}})}-\frac{1}{N_{1}(r_{\vec n+\vec e_{j}}+1)}$ (if $|\vec n+\vec e_{j}-\vec p|_{1}\leq r_{\vec n+\vec e_{j}}$). In the second case we have 
\begin{eqnarray*}
\frac{1}{N_{1}(r_{\vec n+\vec e_{j}})}-\frac{1}{N_{1}(r_{\vec n+\vec e_{j}}+1)}&\leq&\frac{1}{N_{1}(|\vec n+\vec e_{j}-\vec p|_{1})}-\frac{1}{N_{1}(|\vec n+\vec e_{j}-\vec p|_{1}+1)}\\
%&\leq&\frac{1}{N_{1}(|(n+e_{j})_{i}-p_{i}|)}-\frac{1}{(2|(n+e_{j})_{i}-p_{i}|+3)^{d}}\\
&=&\frac{1}{N_{1}(|\vec n-\vec p|_{1}+1)}-\frac{1}{N_{1}(|\vec n-\vec p|_{1}+2)}\\
&<&\frac{1}{N_{1}(|\vec n-\vec p|_{1})}-\frac{1}{N_{1}(|\vec n-\vec p|_{1}+1)}\\
&=&\frac{1}{N_{1}(|\vec n-\vec p|_{1})}-\frac{1}{N(|\vec n+\vec e_{j}-\vec p|_{1})}.
\end{eqnarray*}
%If $j=i$ the third inequality is strict in this case.\\

\smallskip 

\noindent{\it Case 3:} If $\vec n\in X_{j}^{-}$ and $n_{j}< p_{j}$. In this situation we have that the contribution of $f(\vec p)$ to $A_{r_{\vec n}}f(\vec n)-A_{r_{\vec n}+1}f(\vec n+\vec e_{j})$ is non-positive (if $|\vec n-\vec p|_{1}>r_{\vec n}$) or $\frac{1}{N_{1}(r_{\vec n})}-\frac{1}{N_{1}(r_{\vec n}+1)}$ (if $|\vec n-\vec p|_{1}\leq r_{\vec n}$). In the second case we have 
\begin{eqnarray*}
\frac{1}{N_{1}(r_{\vec n})}-\frac{1}{N_{1}(r_{\vec n}+1)}
&\leq&\frac{1}{N_{1}(|\vec p-\vec n|_{1})}-\frac{1}{N_{1}(|\vec p-\vec n|_{1}+1)}\\
&<&\frac{1}{N_{1}(|\vec p-\vec n-\vec e_{j}|_{1})}-\frac{1}{N_{1}(|\vec p-\vec n|_{1})}.
\end{eqnarray*}

\smallskip 

\noindent{\it Case 4:} If  $\vec n\in X_{j}^{+}$ and $n_{j}< p_{j}$. Now we have that the contribution of $f(\vec p)$ to $A_{r_{\vec n+\vec e_{j}}}f(\vec n+\vec e_{j})-A_{r_{\vec n+\vec e_{j}}+1}f(\vec n)$ is non-positive (if $|\vec p-\vec n-\vec e_{j}|_{1}>r_{\vec n+\vec e_{j}}$) or $\frac{1}{N_{1}(r_{\vec n+\vec e_{j}})}-\frac{1}{N_{1}(r_{\vec n+\vec e_{j}}+1)}$ (if $|\vec p-\vec n-\vec e_{j}|_{1}\leq r_{\vec n+\vec e_{j}}$). In the second case we have 
\begin{eqnarray*}
\frac{1}{N_{1}(r_{\vec n+\vec e_{j}})}-\frac{1}{N_{1}(r_{\vec n+\vec e_{j}}+1)}&\leq&\frac{1}{N_{1}(|\vec p-\vec n-\vec e_{j}|_{1})}-\frac{1}{N_{1}(\vec p-\vec n-\vec e_{j}|_{1}+1)}\\
%&\leq&\frac{1}{(2(|p_{i}-(n+e_{j})_{i}|+1)^{d}}-\frac{1}{(2|p_{i}-(n+e_{j})_{i}|+3)^{d}}\\
%&\leq&\frac{1}{(2(|p_{i}-n_{i}|-1)+1)^{d}}-\frac{1}{(2(|p_{i}-n_{i}|-1)+3)^{d}}\\
&=&\frac{1}{N_{1}(|\vec p-\vec n-\vec e_{j}|_{1})}-\frac{1}{N_{1}(|\vec p-\vec n|_{1})}.
\end{eqnarray*}
The equality is achieved if and only if $r_{\vec n+\vec e_{j}}=|\vec p-\vec n-\vec e_{j}|_{1}$.

\smallskip 

\noindent{\it Conclusion:} Given a line $l$ in the lattice, we define the distance from $\vec p$ to $l$ by 
$$
d(l,\vec p)=\min\{|\vec m-\vec p|_{1};\,\vec m\in l\}.
$$
If the direction of $l$ is the same as the direction of $\vec e_{j}$, by intersecting $l$ with the hyperplane $H_{j}=\{\vec z\in\Z^{d}; z_{j}=p_{j}\}$ we obtain the point that realizes the distance from $p$ to $l$. By the previous analysis  we have that the contribution of $f(\vec p)$ to 
$$
\sum_{\vec n\in l\cap X_{j}^{-}}A_{r_{\vec n}}f(\vec n)-A_{r_{\vec n}+1}f(\vec n+\vec e_{j})+\sum_{\vec n\in l\cap X_{j}^{+}}A_{r_{\vec n+\vec e_{j}}}f(\vec n+\vec e_{j})-A_{r_{\vec n+\vec e_{j}}+1}f(\vec n)
$$
is less than or equal to
\begin{equation}\label{cent peso maximo}
\frac{2}{N_{1,d}(d(l,\vec p))}.
\end{equation}
As $p$ belongs to $d$ lines of the lattice, given $k\in \N$ there exist $d(N_{1,d-1}(k)-N_{1,d-1}(k-1))$ lines such that $d(l,\vec p)=k$. Thus the contribution of $f(\vec p)$ to the right-hand side of \eqref{suma d>1} is less than or equal to
$$
\left(2d+\sum_{k\geq 1}\frac{2d(N_{1,d-1}(k)-N_{1,d-1}(k-1))}{N_{1,d}(k)}\right),
$$
and as a consequence of this we obtain the desired inequality. 

\smallskip

If $f$ is a delta function, then there exist $\vec y\in\Z^{d}$ and $k\in\R$ such that 
$$
f(\vec y)=k \ \ \text{and}\ \ \ f(\vec x)=0\ \ \forall \ \vec x\in\Z^{d}\setminus\{y\}.
$$
Considering the contribution of $|f(\vec y)|$ to a line $l$ in the lattice $\Z^{d}$ we have equality in \eqref{cent peso maximo}, and hence in \eqref{eq main theo cent}. On the other hand, let us assume that $f:\Z^{d}\to\R$ is a nonnegative function that verifies the equality in \eqref{eq main theo cent}. We define $P=\{\vec t\in\Z^{d}; f(\vec t)\neq 0\}$ and then
$$
\var M_{1}f=\left(2d+\sum_{k\geq 1}\frac{2d(N_{1,d-1}(k)-N_{1,d-1}(k-1))}{N_{1,d}(k)}\right)\sum_{\vec t\in P}f(\vec t).
$$
Therefore, given $\vec s=(s_{1},s_{2},\dots,s_{d})\in P$ and a line $l$ in the lattice, the contribution of $f(\vec s)$ to $l$ in \eqref{cent peso maximo} must be $\frac{2}{N_{1,d}(d(l,\vec s))}$ by the previous analysis. Then, if there exists $\vec u\in P\setminus\{\vec s\}$, the contribution of $f(\vec u)$ to $l$ in \eqref{suma d>1} must also be $
\frac{2}{N_{1,d}(d(l,\vec u))}$. Assume without loss of generality that $s_{d}>u_{d}$ and consider the line  $l=\{(s_{1},s_{2},\dots,s_{n-1},x); x\in\Z\}$. As we have equality in \eqref{eq main theo cent}, given $\vec n\in l$ such that $n_{d}\geq s_{d}$, we need to have that $\vec n\in X^{-}_{j}$ and 
$|\vec n-\vec s|_{1}=r_{\vec n}=|\vec n-\vec u|_{1}$, which gives us a contradiction. Thus $f$ must be a delta function.

%------------Main Theo Noncent----------------------------------
\section{Proof of Theorem \ref{main theo noncent}} 

\subsection{Preliminaries} As before we start noticing that, since $f\in \ell^{1}(\Z^{d})$, for each $\vec n \in \Z^d$ there exist $r_{\vec n}\in\R^{+}$ and $c_{\vec n}\in\R^{d}$ such that $\vec n\in c_{\vec n}+Q(r_{\vec n})$ and $\wt Mf(\vec n)=A_{r_{\vec n}}f( c_{\vec n})$, where $Q_{r_{\vec n}}=\{m\in\Z^{d}; |m|_{\infty}\leq r_{\vec n}\}=\{m\in\Z^{d},\max\{|m_{1}|,\ldots,|m_{d}|\}\leq r_{\vec n}\}$. 
We now introduce the local maxima and minima of a discrete function $g:\Z \to \R$.\footnote{The local extrema are defined slightly differently in  \cite{BCHP, CH}, but used with the meaning stated here.} We say that an interval $[n,m]$ is a {{\it string of local maxima}} of $g$ if 
$$g(n-1) < g(n) = \ldots = g(m) > g(m+1).$$
If $n = -\infty$ or $m = \infty$ (but not both simultaneously) we modify the definition accordingly, eliminating one of the inequalities. The rightmost point $m$ of such a string is a {\it right local maximum} of $g$, while the leftmost point $n$ is a {\it left local maximum} of $g$. We define {\it string of local minima}, {\it right local minimum} and {\it left local minimum} analogously. 

\smallskip
Given a line $l$ in the lattice $\Z^{d}$ parallel to $\vec e_{d}$ there exists $n'\in\Z^{d-1}$ such that $l=\{(n',m); m\in\Z\}$. Let us assume that $\wt{M}f(n',x)$ is not constant as function of $x$ (otherwise the variation of the maximal function over this line will be zero). Let $\{[a_j^-,a_j^+]\}_{j \in \Z}$ and $\{[b_j^-,b_j^+]\}_{j \in \Z}$ be the ordered strings of local maxima and local minima of $\wt{M}f(n',x)$ (we allow the possibilities of $a_j^{-}$ or $b_j^{-} = - \infty$ and $a_j^{+}$ or $b_j^{+} =  \infty$), i.e. 
\begin{equation}\label{Sec4_sequence}
\ldots < a_{-1}^- \leq a_{-1}^+ < b_{-1}^- \leq b_{-1}^+ < a_0^- \leq a_0^+ < b_0^-\leq b_0^+ < a_1^- \leq a_1^+ < b_1^- \leq b_1^+ < \ldots
\end{equation}
This sequence may terminate in one or both sides and we adjust the notation and the proof below accordingly. Note that we have at least one string of local maxima since $\wt{M}f(\vec n) \to 0$ as $|\vec n|_{\infty} \to \infty$, therefore, if the sequence terminates in one or both sides, it must terminate in a string of local maxima. The variation of the maximal function in $l$ is given by
\begin{equation}\label{sum min max}
2\sum_{j\in\Z}\wt Mf(n',a^{+}_{j})-\wt Mf(n',b^{-}_{j})\leq 2\sum_{j\in\Z}A_{r_{(n',a_{j}^+)}}f( c_{(n',a_{j}^+)})-A_{r_{(n',a_{j}^+)}+|a^{+}_{j}-b^{-}_{j}|}f( c_{(n',a_{j}^+)}).
\end{equation}

\smallskip

We now prove an auxiliary lemma. 
\begin{lemma}\label{a lo mas un}
Given $\vec q\in\Z^{d}$ and a line $l$ in the lattice $\Z^{d}$. There exists at most one string of local maxima of $\wt Mf$ in $l$ such that there exists $\vec n$ in the string whose contribution of $f(\vec q)$ to $A_{r_{\vec n}}f(c_{\vec n})$ is positive.  
\end{lemma}
\begin{proof} Assume without loss of generality that $l=\{(m_{1},m_{2},\dots,m_{d-1}, x); \, x\in\Z\}=\{(m',x);\,x \in\Z\}.$
Consider a string of local maxima of $\wt Mf$ in $l$
\begin{equation}\label{string lem proof}
\wt Mf(m',a-1) < \wt Mf(m',a) = \ldots = \wt Mf(m',a+n) > \wt Mf(m',a+n+1).
\end{equation}
Let
$$
\wt Mf(m',a+i)=A_{r_{(m',a+i)}}f( c_{(m',a+i)})\ \ \forall\ 0\leq i\leq n.
$$
Given $\vec q=(q_{1},q_{2},\dots,q_{d})\in\Z^{d}$, a necessary condition for the contribution of $f(\vec q)$ to $ A_{r_{(m',a+i)}}f( c_{(m',a+i)})$ to be positive for some $i$ is that $a-1< q_{d}< a+n+1$ (otherwise this would violate one of the endpoint inequalities in \eqref{string lem proof}). The result follows from this observation.
\end{proof}

\subsection{Proof of Theorem \ref{main theo noncent}}
Given $\vec p\in\Z^{d}$ and a line $l$ in the lattice $\Z^{d}$, we define $d(l,\vec p)= \min\{|\vec p-\vec m|_{\infty};\, \vec m\in l\}$
and $d(l,\vec p)_{+}=\max\{1,d(l,\vec p)\}$. As consequence of Lemma \ref{a lo mas un}, given $\vec p=(p_{1},p_{2},\dots,p_{d-1},p_{d})\in\Z^{d}$ and a line $l=\{(n_{1},n_{2},\dots,n_{d-1},x)\in\Z^{d}; \, x\in\Z\}$ such that $\big|\{i\in\{1,2,\dots,d-1\};|n_{i}-p_{i}|=d(l,\vec p)\}\big|=j$, the contribution of $f(\vec p)$ to the right-hand side of \eqref{sum min max} is less than or equal to
\begin{equation}\label{dist line}
\frac{2}{(d(l,\vec p)+1)^{j}(d(l,\vec p))_{+}^{d-j}}.
\end{equation}
In fact, if an $\ell^\infty$-cube contains $\vec p$ and a point in $l$ then it must have side at least $d(l,\vec p)$, and it must contain $(d(l,\vec p)+1)$ lattice points in each direction $\vec e_i$ for $i$ such that $|n_{i}-p_{i}|=d(l,\vec p)$. In the other $d-j$ directions the cube contains at least $d(l,\vec p)$ lattice points. This leads to \eqref{dist line}.

\smallskip

If equality in \eqref{dist line} is attained for a point $\vec p$ and a line $l$, then there is a point $\vec q \in l$ that realizes the distance to $\vec p$, belongs to a string of local maxima of $l$, and such that $\vec p\in c_{\vec q}+Q(r_{\vec q})$. Moreover, this string of local maxima must be unique, otherwise $f(\vec p)$ would also have a negative contribution coming from a string of minimum in \eqref{sum min max}. In particular this implies that $\wt Mf(\vec p)\geq \wt Mf(\vec n)$ for all $\vec n\in l$. If we fix a point $\vec p$ and assume that equality in \eqref{dist line} is attained {\it for all lines} $l$ in our lattice, then $\wt Mf(\vec p)\geq \wt Mf(\vec n)$ for all $\vec n\in \Z^d$.

\smallskip
 
Therefore, as $\vec p$ belong to $d$ lines of the lattice $\Z^{d}$, and given $k\in\N$ and $j\in\{1,2,\dots,d-1\}$ there exist $2^{j}{{d-1 \choose j}}(2(k-1)+1)^{d-1-j}$ lines $l=\{(n_{1},n_{2},\dots,n_{d-1},x);\,x\in\Z\}$ such that $d(l,\vec p)=k$ and $\big|\{i\in\{1,2,d\dots,d-1\};|n_{i}-p_{i}|=k\}\big|=j$, the contribution of $f(\vec p)$ 
to the variation of the maximal function in $\Z^{d}$ is less than or equal to
\begin{eqnarray*}
&&2d+d\sum_{k\geq 1}\sum_{j=1}^{d-1}2^{j}{d-1 \choose j}(2k-1)^{d-1-j}\frac{2}{(k+1)^{j}\,k^{d-j}}\\
&&=2d+\sum_{k\geq 1}\frac{2d}{k}\sum_{j=1}^{d-1}{d-1\choose j}\left(\frac{2}{k+1}\right)^{j}\left(\frac{2k-1}{k}\right)^{d-1-j}\\
&&=2d+\sum_{k\geq 1}\frac{2d}{k}\left(\left(\frac{2}{k+1}+\frac{2k-1}{k}\right)^{d-1}-\left(\frac{2k-1}{k}\right)^{d-1}\right).
\end{eqnarray*}
This concludes the proof of \eqref{eq noncent d>1}.

\smallskip

If $f$ is a delta function, with $f(\vec n)=0$ for all $n\in\Z^{d}\setminus\{\vec p\}$ for some $p\in\Z^{d}$, it is easy to see that we have equality in \eqref{dist line} for the contribution of $|f(\vec p)|$ to all lines $l$, which implies equality in \eqref{eq noncent d>1}. On the other hand, let us assume that $f:\Z^{d}\to\R$ is a nonnegative function that verifies the equality in \eqref{eq noncent d>1}. We define $P=\{\vec t\in\Z^{d}; f(\vec t)\neq 0\}$ and thus
$$
\var\wt Mf=\left(2d+\sum_{k\geq 1}\frac{2d}{k}\left(\left(\frac{2}{k+1}+\frac{2k-1}{k}\right)^{d-1}-\left(\frac{2k-1}{k}\right)^{d-1}\right)\right)\sum_{t\in P}f(t).
$$
Then, given $\vec s\in P$, if there exists $\vec u\in P\setminus\{\vec s\}$, we consider a line $l$ in the lattice $\Z^{d}$ such that $\vec s\in l$ and $\vec u\notin l$. The contribution of $f(\vec s)$ to $l$ must be 2, $\wt Mf(\vec s)=f(\vec s)$ belongs to the unique string of local maxima of $\wt Mf$ in $l$ and the right-hand side of \eqref{sum min max}  must be $2f(\vec s)$, by the previous analysis. Therefore the contribution of $f(\vec u)$ to the line $l$ is $0$ and then $f(\vec u)$ does not provide the maximum contribution as predicted in \eqref{dist line}, hence \eqref{eq noncent d>1} cannot attained. We conclude that $f$ must be a delta function.

\section*{Acknowledgents}
\noindent I am deeply grateful to my advisor Emanuel Carneiro for encouraging me to work on this problem, for all the fruitful discussions and for his guidance throughout the preparation of this paper. I would like to thank Renan Finder and Esteban Arreaga for all the interesting discussions related to the proof of Lemma \ref{fund lemma}. I also want to thank Mateus Sousa for a careful review of this paper. The author also acknowledges support from CAPES-Brazil.

\end{document}